\documentclass[10pt]{amsart}

\usepackage{amsmath, amssymb, amscd}
\usepackage{eucal}
\usepackage{mathrsfs}
\usepackage[frame,cmtip,arrow,matrix,line,graph,curve]{xy}
\usepackage{graphicx}

\newcommand{\zz}{\mathbb Z}
\newcommand{\nn}{\mathbb N}

\newtheorem{prop}{Proposition}[section]
\newtheorem{thm}[prop]{Theorem}
\newtheorem{lem}[prop]{Lemma}

\theoremstyle{remark}

\theoremstyle{definition}
 
 \newtheorem{defn}[prop]{Definition}

\numberwithin{equation}{section}

\begin{document}
\title[A note on shadowing properties] {A note on shadowing properties}
\author{Hahng-Yun Chu, Dae Hwan Goo and Se-Hyun Ku$^*$}
\address{Dept. of Mathematics, Chungnam National University, 99 Daehak-ro, Yuseong-gu,
Daejeon 305-764, Republic of Korea}
\email{hychu@cnu.ac.kr\ (H.-Y. Chu)}

\address{Dept. of Mathematics, Chungnam National University, 99 Daehak-ro, Yuseong-gu,
Daejeon 305-764, Republic of Korea}
\email{pi3014@hanmail.net\ (D.-H. Goo)}

\address{Dept. of Mathematics, Chungnam National University, 99 Daehak-ro, Yuseong-gu,
Daejeon 305-764, Republic of Korea}
\email{shku@cnu.ac.kr\ (S.-H. Ku)}

\thanks{\it $\ast$ Corresponding author.}


\begin{abstract}
Let $\mathfrak{X}^{1}(M)$ be the space of $C^{1}$-vector fields on $M$ endowed with the $C^{1}$-topology
and let $\Lambda$ be an isolated set for a $X\in\mathfrak{X}^{1}(M)$.
In this paper, we directly prove that every $X\in\mathfrak{X}^{1}(M)$ having the (asymptotic) average shadowing property
in $\Lambda$ has no proper attractor in $\Lambda$. Our proof is a direct version of the results by Gu and Ribeiro (see ~\cite{Gu2-2009, Gu1-2005, Ri-}).
We also show that every $X\in\mathfrak{X}^{1}(M)$
having the (two-sided) limit shadowing property with a gap in $\Lambda$ is topologically transitive and has the shadowing property in $\Lambda$.

\end{abstract}

\maketitle

\thispagestyle{empty} \markboth{Hahng-Yun Chu, Dae Hwan Goo and Se-Hyun Ku} {A note on shadowing properties}


\section{Introduction}\label{sec: intro}

\bigskip

Let $M$ be a closed $n$-dimensional Riemannian manifold , where $n\geq3$.
We denote by $\mathfrak{X}^{r}(M)$ the set of all $C^{r}$-vector fields on $M$
endowed with the $C^{r}$-topology and denote by $X_{t}$ the flow of $X\in\mathfrak{X}^{r}(M)$.
Recall that $\mathfrak{X}^{r}(M)$ is a Baire space and every resideul subset of $\mathfrak{X}^{r}(M)$ is dense.

Blank introduced the notion of the average shadowing property in studying the chaotic dynamical systems and
showed that certain kinds of perturbed hyperbolic systems have the average shadowing property (see ~\cite{Bl2-1988,
Bl1-1988}).
Later, Gu posed the notion of asymptotic average shadowing property for flows and showed that a flow with the (asymptotic)
average shadowing property is chain transitive (see ~\cite{Gu2-2009, Gu1-2005}). In ~\cite{Ri-}, Riberio showed that
a vector field $X$ is chain transitive in an isolated set $\Lambda$ if and only if $\Lambda$ has no proper attractor for $X$.
He also proved that a $C^{1}$-vector field with the limit shadowing property has no proper attractor.
Combining the results by Gu and Ribeiro, it is proved that the (asymptotic) average shadowing property in an isolated set 
has no proper attractor in the set.

In this article, we directly prove that every $C^{1}$-vector field $X\in\mathfrak{X}^{1}(M)$ having the (asymptotic) average shadowing property in $\Lambda$
has no proper attractor. Recently the authors in ~\cite{CK-2014} introduced the notion of the two-sided limit shadowing property with a gap for discrete dynamical system and studied several properties related to this. We will generalize this notion to vector fields for an isolated set and prove that every $X\in\mathfrak{X}^{1}(M)$
with the (two sided) limit shadowing property with a gap is topologically transitive and has the shadowing property in $\Lambda$.
We say that an invariant set $\Lambda\subset M$ is \textit{isolated} for a vector field $X$
if there exists an open subset $U$ of $M$ such that $\Lambda=\bigcap_{t\in\mathbb{R}}X_t(U)$.
Now we describe the notions of several shadowing properties.

Fix a $\delta>0$. We say that a sequence $\{(x_i, t_i)\}_{i\in\mathbb{Z}}$ of points $x_{i}\in M$
and positive integers $t_{i}$ is a $\delta$-\textit{pseudo orbit} of a vector field $X$
if for all $i\in\mathbb{Z}$ we have $$d(X_{t_i}(x_i),x_{i+1})<\delta.$$
Define the sequence $(s_n)_{n\in\mathbb{Z}}$ by

$$s_n=\left\{
   \begin{array}{ll}
     0, & n=0\hbox{;} \\
     \sum_{i=0}^{i=n-1}t_i, & n\in\mathbb{N}\hbox{;} \\
     \sum_{i=n}^{-1}t_i, & -n\in\mathbb{N}\hbox{.}
   \end{array}
 \right.
$$

A $\delta$-pseudo orbit $\{(x_i, t_i)\}_{i\in\mathbb{Z}}$ of $X$ is
$\varepsilon$-\textit{shadowed} by the orbit through a point $x$,
if there is an orientation preserving homeomorphism $h:\mathbb{R}\rightarrow\mathbb{R}$
with $h(0)=0$ such that
$$d(X_{h(t)}(x),X_{t-s_i}(x_{i}))<\varepsilon \ \mbox{for all $s_i\leq t\leq s_{i+1}$}, i\in\mathbb{Z}.$$

Let $\Lambda$ be an isolated set for $X\in\mathfrak{X}^{1}(M)$.
We say that the vector field $X$ has the \textit{shadowing property in} $\Lambda$,
if for any $\varepsilon>0$ there is a $\delta>0$ such that every $\delta$-pseudo orbit of $X$ in $\Lambda$ is
$\varepsilon$-shadowed by some real orbit of $X$ in $\Lambda$.
If $\Lambda=M$, then we say that $X$ has the \textit{shadowing property}.
A vector field $X$ is called to be \textit{chain transitive in} $\Lambda$ if for any $x,y\in\Lambda$ and any $\delta>0$
there is a finite $\delta$-pseudo orbit $\{(x_i, t_i)\}_{0\leq i\leq K}$ of $X$ in $\Lambda$
such that $x_{0}=x$ and $x_{K}=y$, where $K\in{\mathbb Z}_{+}$.
If $\Lambda=M$, then we say that $X$ is \textit{chain transitive}.

A sequence $\{(x_i, t_i)\}_{i\in\mathbb{Z}}$ is called a $\delta$-\textit{average-pseudo orbit} of a vector field $X$
if $t_i\geq1$ for every $i\in\mathbb{Z}$ and there exists a positive integer $N$ such that for any $n\geq N$ and $k\in\mathbb{Z}$ it satisfies
$$\frac{1}{n}\sum_{i=1}^{n}d(X_{t_{i+k}}(x_{i+k}),x_{i+k+1})<\delta.$$

A $\delta$-average-pseudo orbit $\{(x_i, t_i)\}_{i\in\mathbb{Z}}$ of a vector field $X$
is called to be \textit{positively (resp. negatively) $\varepsilon$-shadowed in average} by the orbit of $X$ through a point $x$,
if there is an orientation preserving homeomorphism $h:\mathbb{R}\rightarrow\mathbb{R}$ with $h(0)=0$
such that
$$\limsup_{n\rightarrow\infty} \frac{1}{n}\sum_{i=1}^{n}\int_{s_i}^{s_{i+1}}d(X_{h(t)}(x),X_{t-s_{i}}(x_{i}))dt<\varepsilon$$

$$\left(\text{ resp. }\limsup_{n\rightarrow\infty} \frac{1}{n}\sum_{i=1}^{n}\int_{-s_{-i}}^{-s_{-(i-1)}}d(X_{h(t)}(x),X_{t+s_{-i}}(x_{-i}))dt<\varepsilon\right),$$
where $s_0=0, s_n=\sum_{i=0}^{n-1}, s_{-n}=\sum_{i=-n}^{-1}t_i, n\in\mathbb{N}.$

Given an isolated set $\Lambda$ of a vector field $X\in\mathfrak{X}^{1}(M)$,
we say that $X$ has the \textit{average shadowing property in} $\Lambda$,
if for any $\varepsilon>0$ there is a $\delta>0$ such that every $\delta$-average-pseudo orbit of $X$ in $\Lambda$ is both positively and negatively $\varepsilon$-shadowed in average by some orbit of $X$ in $\Lambda$.
If $\Lambda= M$, then we say that $X$ has the \textit{average shadowing property}.

Now we consider a double-infinite sequence $(x_i)_{i\in{\mathbb Z}}$ to introduce the two-sided limit shadowing property.
A sequence $\{(x_i, t_i)\}_{i\in\mathbb{Z}}$ is called a \textit{(two-sided) limit-pseudo orbit} of a vector field $X$,
if $t_i\geq 1$ for every $i\in\mathbb{Z}$ and $$\lim_{|i|\rightarrow\infty}d(X_{t_i}(x_i), x_{i+1})=0.$$

A limit-pseudo orbit $\{(x_i, t_i)\}_{i\in\mathbb{Z}}$ of $X$ is called to be \textit{positively (resp. negatively) shadowed in limit}
by an orbit of $X$ through a point $x$,
if there is an orientation preserving homeomorphism $h:\mathbb{R}\rightarrow\mathbb{R}$ with $h(0)=0$
such that
$$\lim_{i\rightarrow\infty} \int_{s_i}^{s_{i+1}}d(X_{h(t)}(x),X_{t-s_{i}}(x_{i}))dt=0$$

$$\left(\text{ resp. }\lim_{i\rightarrow\infty} \int_{-s_{-i}}^{-s_{-(i-1)}}d(X_{h(t)}(x),X_{t+s_{-i}}(x_{-i}))dt=0\right),$$
where $s_0=0, s_n=\sum_{i=0}^{n-1}, s_{-n}=\sum_{i=-n}^{-1}t_i, n\in\mathbb{N}.$

Given an isolated set $\Lambda$ of $X\in\mathfrak{X}^{1}(M)$,
we say that $X$ has the \textit{limit shadowing property in} $\Lambda$,
if every limit-pseudo orbit in $\Lambda$ is both positively and negatively shadowed in limit by an orbit of $X$ in $\Lambda$.
If $\Lambda=M$, then we say that $X$ has the \textit{limit shadowing property}.

In~\cite{Gu2-2009}, Gu introduced the notion of the asymptotic average shadowing property for flows
which is a generalization of the limit shadowing property in random dynamical systems.
A sequence $\{(x_i, t_i)\}_{i\in\mathbb{Z}}$ is called an \textit{asymptotic average-pseudo orbit} of $X$,
if $t_i\geq1$ for every $i\in\mathbb{Z}$ and
$$\lim_{n\rightarrow\infty}\frac{1}{n}\sum_{i=-n}^{n}d(X_{t_i}(x_i),x_{i+1})=0.$$

An asymptotic average-pseudo orbit $\{(x_i, t_i)\}_{i\in\mathbb{Z}}$ of $X$ is called to be
\textit{positively (resp. negatively) asymptotically shadowed in average} by an orbit of $X$ through $x$,
if there is an orientation preserving homeomorphism $h:\mathbb{R}\rightarrow\mathbb{R}$ with $h(0)=0$
such that
$$\lim_{n\rightarrow\infty} \frac{1}{n}\sum_{i=0}^{n}\int_{s_i}^{s_{i+1}}d(X_{h(t)}(x),X_{t-s_{i}}(x_{i}))dt=0$$
$$\left(\text{ resp. }\lim_{n\rightarrow\infty} \frac{1}{n}\sum_{i=0}^{n}\int_{-s_{-i}}^{-s_{-(i-1)}}d(X_{h(t)}(x),X_{t+s_{-i}}(x_{-i}))dt=0\right),$$
where $s_0=0, s_n=\sum_{i=0}^{n-1}, s_{-n}=\sum_{i=-n}^{-1}t_i, n\in\mathbb{N}$.

Given an isolated set $\Lambda$ of $X\in\mathfrak{X}^{1}(M)$,
we say that $X$ has the \textit{asymptotic average shadowing property in} $\Lambda$,
if every asymptotic average-pseudo orbit in $\Lambda$ is both positively and negatively asymptotically shadowed in average by some real orbit of $X$ in $\Lambda$.
If $\Lambda=M$, then we say that $X$ has the \textit{asymptotic average shadowing property}.

In this article, we prove that several shadowing properites imply non-existence of an attractor for a given $C^1$-vector field.
In addition, we focus on the notion of the two-sided limit shadowing property with a gap.
The concepts of the topological transitivity and shadowing property follow the notion.

\bigskip

\bigskip


\section{Main Results}\label{sec: def}

\bigskip

In this paper we consider a vector field $X\in\mathfrak{X}^{1}(M)$, where $M$ is a closed Riemannian manifold of
dimension $n\geq 3$ and $\Lambda$ is an isolated set of $M$ which is not a periodic orbit or a singularity.

Let $\Lambda$ be a compact invariant subset of $M$ and $X\in\mathfrak{X}^{1}(M).$
The \textit{basin} $B(A)$ of $A$ with respect to a vector field $X$ is defined by
the set $\{x\in M | \omega_X(x)\subseteq A\}$.
We say that a compact invariant set $\Lambda$ is \textit{attracting}
if there is an open set $U$ such that $X_t(\overline{U})\subset U$ for all $t>0$ and $\Lambda=\bigcap_{t\geq0}X_t(\overline{U})$.
An \textit{attractor} of $X$ is a transitive attracting set of $X$. An attractor for $-X$ is called a \textit{repeller}.
If $\emptyset\neq\Lambda\subseteq B(\Lambda)\subsetneq M$, then $\Lambda$ is called a \textit{proper} attractor or repeller.
A \textit{sink (source)} of $X$ is a attracting (repelling) critical orbit of $X$.
A point $x\in M$ is called a \textit{chain recurrent point}
if for any $\epsilon>0$ there exists an $\epsilon$-pseudo orbit of $X$ from $x$ to $x$.
A subset $A\subset M$ is called a \textit{chain recurrent set} if any $x\in A$ is a chain recurrent point.

It is obvious that chain transitivity implies non-existence of sinks and sources in $\Lambda$.
In~\cite{Ri-}, Ribeiro showed that if a $C^1$-vector field has the limit shadowing property in an isolated subset of a closed manifold $M$ then it has no proper attractor. He also showed that an invariant set is chain transitive if and only if it does not contain proper attractors.

In this section, we first deal with the notions of several shadowing properties related to the non-existence of attractors.
We focus on the notions of the average shadowing property, the asymptotic average shadowing property and the two-sided limit shadowing property with a gap.

\smallskip

\begin{prop}\label{chnoatt}~\cite[Proposition 3]{Ri-}
A vector field $X$ is chain transitive in an isolated set $\Lambda$ if and only if
$\Lambda$ has no proper attractor for $X$.
\end{prop}

\smallskip

\begin{thm}\label{th:asp}
Let $\Lambda$ be an isolated set of $X$. If $X$ has the average shadowing property in $\Lambda$
then $\Lambda$ has no proper attractor for $X$.
\end{thm}
\begin{proof}
Assume that there is a proper attractor $A \subset \Lambda$.
Then $A\neq \emptyset$ and $\Lambda\setminus B(A)\neq \emptyset$.
Since $A$ is an attractor,
we can take an attracting $\frac{\epsilon_{0}}{2}$-neighborhood $U(A,\frac{\epsilon_{0}}{2})$ of $A$
such that $U(A,\epsilon_{0})\subset B(A)$.
Choose $b\in \Lambda\setminus B(A)$ and $a\in A$.
Consider the sequence $\{(x_i, t_i)\}_{i\in\mathbb{Z}}$ constructed as follows:
\begin{align*}
&x_{i}=X_{i}(a),\quad t_{i}=1,\quad i\leq 0\\
&x_{i}=X_{i}(b),\quad t_{i}=1,\quad i>0,
\end{align*}
with $i\in\zz$.
Let $\delta>0$ be given.
Fix a sufficient large integer $N_{\delta}\in\nn$ such that $\frac{d(a,X_{1}(b))}{N_{\delta}}<\delta$.
It is easy to see that for any $n\geq\ N_{\delta}$ and any $k\in\zz$
$$\frac{1}{n}\sum_{i=1}^{n}d(X_{t_{i+k}}(x_{i+k}),x_{i+k+1})\leq \frac{d(a,X_{1}(b))}{n}\leq \frac{d(a,X_{1}(b))}{N_{\delta}}<\delta.$$
Thus the above sequence $\{(x_i, t_i)\}_{i\in\mathbb{Z}}$ is a $\delta$-average-pseudo orbit of $X$ in $\Lambda$.

Suppose that the sequence $\{(x_i,t_i)\}$ can be both positively and negatively
$\frac{\epsilon_{0}}{2}$-shadowed in average by the orbit of $X$ through some point $z\in \Lambda$,
that is, there is an orientation preserving homeomorphism
$h:\mathbb{R}\rightarrow\mathbb{R}$ with $h(0)=0$ such that
$$\limsup_{n\rightarrow\infty} \frac{1}{n}\sum_{i=1}^{n}\int_{i}^{i+1}d(X_{h(t)}(z),X_{t-i}(x_{i}))dt<\frac{\epsilon_0}{2}$$
and
$$\limsup_{n\rightarrow\infty} \frac{1}{n}\sum_{i=1}^{n}\int_{-i}^{-i+1}d(X_{h(t)}(z),X_{t+i}(x_{-i}))dt<\frac{\epsilon_0}{2}.$$
Thus we have the following two facts $(\dag)$:

(1) There are infinitely many positive integers $i$ such that
$$d(X_{h(t_{i}^{*})}(z),X_{t_{i}^{*}+i}(x_{-i}))<\frac{\epsilon_{0}}{2} \quad \mbox{for some}\,\, t_{i}^{*}\in [-i,-i+1]. $$

(2) There are infinitely many positive integers $i$ such that
$$d(X_{h(t_{i}^{*})}(z),X_{t_{i}^{*}-i}(x_{i}))<\frac{\epsilon_{0}}{2} \quad \mbox{for some}\,\, t_{i}^{*}\in [i,i+1].$$
If (1) is not true, there is a positive integer $N_{0}$ such that for all $i\geq N_{0}$,
$d(X_{h(t)}(z),X_{t+i}(x_{-i}))\geq\frac{\epsilon_{0}}{2}$ for any $t\in[-i,-i+1]$.
\noindent
So,
$$\int_{-i}^{-i+1}d(X_{h(t)}(z),X_{t+i}(x_{-i}))dt\geq\frac{\epsilon_{0}}{2} \quad \mbox{for all}\,\, i\geq N_0.$$
Thus,$$\limsup_{n\rightarrow\infty} \frac{1}{n}\sum_{i=1}^{n}\int_{-i}^{-i+1}d(X_{h(t)}(z),X_{t+i}(x_{-i}))dt\geq\frac{\epsilon_{0}}{2}.$$
This is a contraction. Hence (1) holds. The proof of (2) is similar to that of (1).

Since $A$ is an $X_{t}$ invariant set, by (1), $X_{h(t_{i}^{*})}(z)\in U(A,\frac{\epsilon_{0}}{2})$.
Thus, $X_{t}(X_{h(t_{i}^{*})}(z))\in U(A,\frac{\epsilon_{0}}{2})$ for all $t>0$.
Here we can take $t=-h(t_{i}^{*})$.
Then $z\in U(A,\frac{\epsilon_{0}}{2})$.
Since $b\in \Lambda-B(A)$, the positive orbit of $b$ of $X$ is in the $\Lambda-B(A)$.
Hence we have that
$d(X_{h(t)}(z),X_{t-i}(x_{i}))\geq\frac{\epsilon_{0}}{2} $ for all $i\leq t\leq i+1$.
This implies that
$$\limsup_{n\rightarrow\infty} \frac{1}{n}\sum_{i=1}^{n}\int_{i}^{i+1}d(X_{h(t)}(z),X_{t-i}(x_{i}))dt\geq\frac{\epsilon_{0}}{2}.$$
This is a contraction, which finishes the proof.
\end{proof}

\smallskip

\begin{thm}\label{th:aasp}
Let $\Lambda$ be an isolated set of $X$. If $X$ has the aspmptotic average shadowing property in $\Lambda$
then $\Lambda$ has no proper attractor $X$.
\end{thm}
\begin{proof}
Assume that there is a proper attractor $A \subset \Lambda$.
Then $A\neq \emptyset$ and $\Lambda\setminus B(A)\neq \emptyset$.
Since $A$ is an attractor, we can take an attracting $\frac{\epsilon_{0}}{2}$-neighborhood $U(A,\frac{\epsilon_{0}}{2})$ of $A$
such that $U(A,\epsilon_{0})\subset B(A)$.
Choose $b\in \Lambda\setminus B(A)$ and $a\in A$.
Consider the sequence $\{(x_i, t_i)\}_{i\in\mathbb{Z}}$ constructed as follows:
\begin{align*}
&x_{i}=X_{i}(a),\quad t_{i}=1,\quad i\leq 0\\
&x_{i}=X_{i}(b),\quad t_{i}=1,\quad i>0,
\end{align*}
with $i\in\zz$.
It is easy to see that
$$\frac{1}{n}\sum_{i=-n}^{n}d(X_{t_i}(x_{i}),x_{i+1})\leq \frac{d(a,X_{1}(b))}{n}.$$
So,
$$\lim_{n\rightarrow\infty}\frac{1}{n}\sum_{i=-n}^{n}d(X_{t_i}(x_i),x_{i+1})=0 .$$
Thus the above sequence $\{(x_i, t_i)\}_{i\in\mathbb{Z}}$ is an asymptotic average-pseudo orbit of $X$ in $\Lambda$.
Hence it can be both positively and negatively asymptotically shadowed in average by an orbit of $X$ through some point $z\in\Lambda$
,that is, there is an orientation preserving homeomorphism
$h:\mathbb{R}\rightarrow\mathbb{R}$ with $h(0)=0$ such that
$$
\lim_{n\rightarrow\infty} \frac{1}{n}\sum_{i=1}^{n}\int_{i}^{i+1}d(X_{h(t)}(z),X_{t-i}(x_{i}))dt=0
$$
and
$$
\lim_{n\rightarrow\infty} \frac{1}{n}\sum_{i=1}^{n}\int_{-i}^{-i+1}d(X_{h(t)}(z),X_{t+i}(x_{-i}))dt=0.
$$
Then we also have the facts $(\dag)$ in Theorem \ref{th:asp}. Therefore, by following the same procedure in the proof of Theorem \ref{th:asp},
we can obtain
%
%
%
$$
\lim_{n\rightarrow\infty} \frac{1}{n}\sum_{i=1}^{n}\int_{i}^{i+1}d(X_{h(t)}(z),X_{t-i}(x_{i}))dt\geq\frac{\epsilon_{0}}{2}.
$$
This is a contraction.
\end{proof}

\smallskip

We say that a sequence $\{(x_i, t_i)\}_{i\in\mathbb{N}}$ is a \textit{positive limit-pseudo orbit} of $X$,
if $t_i\geq 1$ for every $i\in\mathbb{N}$ and
$$\lim_{i\rightarrow\infty}d(X_{t_i}(x_i), x_{i+1})=0.$$
Given an isolated set $\Lambda$ of $X\in\mathfrak{X}^{1}(M)$,
the vector field $X$ has the \textit{positive limit shadowing property in} $\Lambda$,
if every positive limit-pseudo orbit in $\Lambda$ is positively shadowed in limit by an orbit of $X$ in $\Lambda$.
If $\Lambda=M$, then we say that $X$ has the \textit{positive limit shadowing property}.

\smallskip

\begin{defn}
A (two-sided) limit-pseudo orbit $\{(x_i, t_i)\}_{i\in\mathbb{Z}}$ of $X$ is \textit{(two-sided) limit shadowed with gap $K\in\mathbb{R}$ for $X$}
if there exists a point $z\in M$ satisfying
\begin{eqnarray*}
&&\lim_{i\rightarrow\infty} \int_{s_i}^{s_{i+1}}d(X_{t+K}(z),X_{t-s_{i}}(x_{i}))dt=0,\\
&&\lim_{i\rightarrow\infty} \int_{-s_{-i}}^{-s_{-(i-1)}}d(X_{t}(z),X_{t+s_{-i}}(x_{-i}))dt=0.
\end{eqnarray*}
where $s_{0}=0, s_n=\sum_{i=0}^{n-1}t_i, s_{-n}=\sum_{i=-n}^{-1}t_i, n\in\mathbb{N}$.

Let $\Lambda$ be an isolated set of $X\in\mathfrak{X}^{1}(M)$.
For $N\in[0,\infty),$ we say that $X$ has the \textit{(two-sided) limit shadowing property with gap $N$ in $\Lambda$}
if every (two-sided) limit pseudo-orbit of $X$ in $\Lambda$ is (two-sided) limit shadowed
with gap $K\in\mathbb{R}$ with $|K|\leq N$.
If such an $N\in[0,\infty)$ exists,
we also say that $X$ has the \textit{(two-sided) limit shadowing property with a gap in $\Lambda$}.
\end{defn}

\smallskip

It is easily showed that the (two-sided) limit shadowing property with gap $0$ implies
the notion of (two-sided) limit shadowing property.
To prove the non-existence of an attractor, we need lemmas.
From the definition of the shadowing property with a gap,
it is obvious that the property implies the positive limit shadowing property.

\smallskip

\begin{lem}\label{lem:plsp}
Let $\Lambda$ be an isolated set of $X\in\mathfrak{X}^{1}(M)$. If $X$ has the (two-sided) limit shadowing property with a gap in $\Lambda$,
then $X$ and $-X$ have the positive limit shadowing property.
\end{lem}

\smallskip

For each $x\in M$ the \textit{stable set} of $x$ is the set
$$W^{s}(x)=\{y\in M |\lim_{t\rightarrow\infty}d(X_{t}(x), X_{t}(y))=0 \},$$
and the \textit{unstable set} of $x$ is the set
$$W^{u}(x)=\{y\in M |\lim_{t\rightarrow\infty}d(X_{-t}(x), X_{-t}(y))=0 \}.$$


\smallskip

Next lemma plays an important role to prove the non-existence of an attractor.

\smallskip

\begin{lem}\label{lem:nonempty}
Let $\Lambda$ be an isolated set of $X\in\mathfrak{X}^{1}(M)$.
If $X$ has the (two-sided) limit shadowing property with gap $N$ in $\Lambda$,
then for every $x,y\in \Lambda$ there is a $K\in\mathbb{R}$, with $|K|\leq N$, satisfying
$$W^{s}(X_{-K}(x))\cap W^{u}(y)\neq \emptyset.$$
\end{lem}
\begin{proof}
Consider the following sequence:
\begin{eqnarray*}
z_{i}=X_{i}(x),\quad t_{i}=1,\quad i\geq0\\
 z_{i}=X_{i}(y),\quad t_{i}=1,\quad i<0
\end{eqnarray*}
Clearly the sequence $\{(z_i, t_i)\}_{i\in\mathbb{Z}}$ is a (two-sided) limit-pseudo orbit of $X$.
Thus there exist $z\in \Lambda$ and $K\in[-N,N]$ satisfying
$$\lim_{i\rightarrow\infty} \int_{i}^{i+1}d(X_{t+K}(z),X_{t-i}(z_{i}))dt=\lim_{i\rightarrow\infty} \int_{-i}^{-(i-1)}d(X_{t}(z),X_{t+i}(z_{-i}))dt=0.$$
That is,
$$\lim_{i\rightarrow\infty} \int_{i}^{i+1}d(X_{t+K}(z),X_{t}(x))dt=\lim_{i\rightarrow\infty} \int_{-i}^{-(i-1)}d(X_{t}(z),X_{t}(y))dt=0.$$
Thus, by the continuity of distance, we have that
$$ \lim_{t\rightarrow\infty} d(X_{t}(X_{K}(z)),X_{t}(x))=\lim_{t\rightarrow-\infty} d(X_{t}(z),X_{t}(y))=0.$$
This implies that $X_K(z)\in W^{s}(x)$ and $z\in W^{u}(y)$.
Hence, $z\in W^{s}(X_{-K}(x))\cap W^{u}(y)\neq \emptyset$.
\end{proof}

\smallskip

\begin{thm}\label{thm:tlspsp}
Let $\Lambda$ be an isolated set of $X\in\mathfrak{X}^{1}(M)$.
If $X$ has the the (two-sided) limit shadowing property with gap $N$ in $\Lambda$,
then $\Lambda$ has no proper attractor for $X$.
\end{thm}
\begin{proof}
Assume that there is a proper attractor $A \subset \Lambda$.
Then there is a neighborhood $U$ of $A$ such that
$A=\bigcap_{t\geq0}X_t(\overline{U})$ and $X_t(\overline{U})\subset U$.
Let $B=\bigcap_{t\geq0}X_{-t}(\Lambda - \overline{U}).$
Choose $a\in A$ and $b\in B$. By Lemma~\ref{lem:nonempty}, there is a $z\in W^{s}(X_{-K}(b))\cap W^{u}(a)$.
Since $ \lim_{t\rightarrow\infty} d(X_{t}(z),X_{t}(X_{-K}(b)))=0$ and $ \lim_{t\rightarrow-\infty} d(X_{t}(z),X_{t}(a))=0$,
there is a sufficiently large $T>0$ such that $X_{-T}(z)\in U$ and $X_{T}(z)\in \Lambda - \overline{U}$.
This contradicts the definition of $U$.
\end{proof}

\smallskip

Finally, we characterize the notion of the (two-sided) limit shadowing property with a gap.
The property is stronger than the notions of transitivity and shadowing property as we shall see.
Let $\Lambda$ be an isolated set of $X\in\mathfrak{X}^{1}(M)$. We say that $X$ is \textit{topologically transitive in $\Lambda$}
if for any nonempty open subsets $U$ and $V$ of $M$ there is some $t\in\mathbb{R}$ such that $X_t(U)\cap V\neq\varnothing$.
If $\Lambda=M$ we simply say that $X$ is \textit{topologically transitive}.
From the definitions, the following lemma is directly proved.

\smallskip

\begin{lem}\label{lem:chsptt}
Let $\Lambda$ be an isolated set of $X\in\mathfrak{X}^{1}(M)$. If $X$ is chain transitive and has the shadowing property in $\Lambda$,
then $X$ is topologically transitive in $\Lambda$.
\end{lem}

\smallskip

We obtain the following theorem which expresses the relation between the notions of chain transitivity, limit shadowing property and shadowing property in a flow version.
The next proposition is essential to prove the last theorem which is dealt with the notion of limit shadowing property with a gap.
Following the proof of Theorem 7.3 in ~\cite{KKO-2014} we prove the following.

\smallskip

\begin{prop}\label{thm:chplspsp}
Let $\Lambda$ be an isolated set of $X\in\mathfrak{X}^{1}(M)$.
If $X$ is chain transitive and has the positive limit shadowing property in $\Lambda$, then $X$ has the shadowing property.
\end{prop}
\begin{proof}
Suppose on the contrary that $X$ does not have shadowing. Then there is $\epsilon>0$ such that for any $n\in\mathbb{N}$
there is a finite $\frac{1}{n}$-pseudo orbit $\alpha_{n}=\{(x_{n,i}, t_{n,i}) \}$ which cannot be $\epsilon$-shadowed by
any point in $\Lambda$. Using chain transitivity, for every $n$ there exists a $\frac{1}{n}$-pseudo orbit
$\beta_{n}=\{(y_{n,i}, \overline{t_{n,i}}) \}$ such that the sequence $\alpha_n \beta_n \alpha_{n+1}$ forms a finite
$\frac{1}{n}$-pseudo orbit. Then the infinite concatenation
$$
\alpha_1 \beta_1 \alpha_2 \beta_2 \alpha_3 \beta_3 \ldots
$$
is a positive limit-pseudo orbit denoted by $\{(x_i, t_i)\}_{i\in\mathbb{N}}$. By positive limit shadowing property,
it is positively $\epsilon$-shadowed by some point $z\in\Lambda$. So, there is an orientation preserving homeomorphism
$h:\mathbb{R}\rightarrow\mathbb{R}$ with $h(0)=0$ such that
$$
\lim_{i\rightarrow\infty} \int_{s_i}^{s_{i+1}}d(X_{h(t)}(x),X_{t-s_{i}}(x_{i}))dt=0,
$$
where $s_{0}=0, s_n=\sum_{i=0}^{n-1}t_i, s_{-n}=\sum_{i=-n}^{-1}t_i, n\in\mathbb{N}$.

Therefore there is $N\in\mathbb{N}$ such that for every $i\geq N$,
$$
\int_{s_i}^{s_{i+1}}d(X_{h(t)}(z),X_{t-s_{i}}(x_{i}))dt<\epsilon.
$$
Thus, for each $i\geq N$, we have
$$
d(X_{h(t)}(z),X_{t-s_{i}}(x_{i}))<\epsilon \quad \mbox{for all $ s_i\leq t \leq s_{i+1}$ }.
$$
This means that there is a finite pseudo orbit $\alpha_n$ which is $\epsilon$-shadowed by some point, which is a contradiction.
\end{proof}

\smallskip
Now we are ready to prove the following theorem.
\begin{thm}
Let $\Lambda$ be an isolated set of $X\in\mathfrak{X}^{1}(M)$.
If $X$ has the the (two-sided) limit shadowing property with a gap in $\Lambda$, then $X$ is topologically transitive and
has the shadowing property in $\Lambda$.
\end{thm}
\begin{proof}
By Lemma~\ref{lem:plsp} $X$ has the positive limit shadowing property.
Also, Proposition~\ref{chnoatt} and Theorem~\ref{thm:tlspsp} say that $X$ is chain transitive.
Then, by Lemma~\ref{lem:chsptt} and Proposition~\ref{thm:chplspsp},
we conclude that $X$ is topologically transitive and has the shadowing property in $\Lambda$.
\end{proof}





\bigskip

\bigskip


\section{Acknowledgements}

\bigskip

This research has been performed as a subproject of project Research for Applications of Mathematical Principles (No.C21501) and supported by the National Institute of Mathematics Sciences (NIMS).

\end{document}